\documentclass{article}

\usepackage[utf8]{inputenc}
\usepackage[english]{babel}
\usepackage{amsmath}
\usepackage{amsthm}
\usepackage{amsfonts}
\usepackage{amssymb}
\usepackage{nccmath}
\usepackage{makecell}
\usepackage{braket}
\usepackage{bussproofs}

\usepackage[capitalise]{cleveref}

\usepackage{tikz}

\newcommand{\CS}{\mathsf{CS}}
\newcommand{\APP}{\mathsf{APP}}
\newcommand{\Tm}{\mathsf{Tm}}
\newcommand{\ATm}{\mathsf{ATm}}
\newcommand{\TmA}{\mathsf{Tm}^{\mathsf{A}}}
\newcommand{\TmP}{\mathsf{Tm}^{\mathsf{P}}}

\newcommand{\Prop}{\mathsf{Prop}}
\newcommand{\Logic}{\mathsf{L}}
\newcommand{\LJstar}{\mathsf{L^\star}}
\newcommand{\LogicAPP}{\mathsf{L^\mathsf{A}}}
\newcommand{\LJCSapp}{\mathsf{L^\mathsf{A}_{CS}}}
\newcommand{\Lcp}{\mathsf{L_{cp}}}
\newcommand{\CL}{\mathsf{CL}}
\newcommand{\LJCS}{\mathsf{L^{\star}_{CS}}}
\newcommand{\LJprob}{\mathsf{L_{prob}}}
\newcommand{\Lstarcomp}{\mathsf{L^{\star}_\alpha}}
\newcommand{\Lstarvol}{\mathsf{L^{\star}_\beta}}
\newcommand{\Lappcomp}{\mathsf{L^\mathsf{A}_\alpha}}
\newcommand{\Lappvol}{\mathsf{L^\mathsf{A}_\beta}}
\newcommand{\LJ}{\mathcal{L}_J}
\newcommand{\LJA}{\mathcal{L}_J^{\mathsf{A}}}

\newcommand{\cstar}{\mathsf{c}^\star}
\newcommand{\cstarterm}{\mathsf{c}^\star\textbf{-term}}
\newcommand{\cstarterms}{\mathsf{c}^\star\textbf{-terms}}

\newcommand{\model}{\mathcal{M}}
\newcommand{\WMP}{W_{MP}}
\newcommand{\WMPC}{W_{MP}^C}
\newcommand{\WnullC}{W_0^C}

\newcommand{\powerset}{\mathcal{P}}
\newcommand{\axjplus}{\textbf{j+}}
\newcommand{\axjcstar}{\textbf{jc\textsuperscript{$\star$}}}
\newcommand{\axj}{\textbf{j}}
\newcommand{\axjfour}{\textbf{j4}}
\newcommand{\axjd}{\textbf{jd}}
\newcommand{\axjt}{\textbf{jt}}
\newcommand{\axcl}{\textbf{cl}}

\newcommand{\PE}{\mathsf{PE}}
\newcommand{\AEGX}{\mathsf{AE^\Gamma(X)}}

\makeatletter
\newcommand{\xLongLeftRightArrow}[2][]{\ext@arrow0055{\LongLeftRightArrowfill@}{#1}{#2}}
\def\LongLeftRightArrowfill@{\arrowfill@\Leftarrow\Relbar\Rightarrow}
\makeatother
\makeatletter
\def\th@plain{%
	\thm@notefont{}
	\itshape 
}

\makeatletter
\newcommand{\todo}[1]{\marginpar{\textbf{TODO\footnotemark}}\@latex@warning{TODO: #1}\footnotetext{ #1}}
\makeatother

\def\th@Definition{%
	\thm@notefont{}
	\normalfont 
}
\makeatother

\newtheorem{thm}{Theorem} 
\newtheorem{lem}[thm]{Lemma} 
\newtheorem{defn}[thm]{Definition} 
\newtheorem{remark}[thm]{Remark}

\begin{document}

\author{Eveline Lehmann \and Thomas Studer}

\title{Subset models for justification logic\thanks{This work was supported by the Swiss National Science Foundation grant 200021$\_$165549.}}

\maketitle

\begin{abstract}
	We introduce a new semantics for justification logic based on subset relations. Instead of using the established and more symbolic interpretation of justifications, we model justifications as sets of possible worlds. 
	We introduce a new justification logic that is sound and complete with respect to our semantics. Moreover, we present another variant of our semantics that corresponds to traditional justification logic.
	
	These types of models offer us a versatile tool to work with justifications, e.g.~by extending them with a probability measure to capture uncertain justifications. Following this strategy we will show that they subsume Artemov's approach to aggregating probabilistic evidence.
\end{abstract}

\section{Introduction}

Justification logic is a variant of modal logic that includes terms representing explicit evidence. A formula of the form $t:A$ means that \emph{$t$ justifies $A$} (or \emph{$t$ represents evidence for $A$}, or \emph{$t$ is a proof of $A$}). Justification logic has been introduced by Artemov~\cite{Artemov1995OperationalModalLogic,Artemov2001ExplicitProvability} to give a classical provability interpretation to $\mathsf{S4}$. Later it turned out that this approach is not only useful in proof theory~\cite{Artemov2001ExplicitProvability,KSweak} but also in epistemic logic~\cite{Art06TCS,Art08RSL,BucKuzStu11JANCL,BucKuzStu14Realizing}.
For a general overview on justification logic, we refer to~\cite{ArtFit11SEP,artemovFittingBook,justificationLogic}.

There are various kinds of semantics available for justification logic. Most of them interpret justification terms in a symbolic way. In provability interpretations~\cite{Artemov2001ExplicitProvability,KSweak}, terms represent (codes of) proofs in formal system like Peano arithmetic. In Mkrtychev models~\cite{Mkr97LFCS}, which are used to obtain decidability, terms are represented as sets of formulas. In Fitting models~\cite{Fit05APAL}, the evidence relation maps pairs of terms and possible worlds to sets of formulas. In modular models~\cite{Art12SLnonote,KuzStu12AiML}, the logical type a justification is a set of formulas, too. Notable exceptions are~\cite{artemov2016onAggregatingPE,Artemov2008TopoJL} where terms are interpreted as sets of possible worlds. However, these papers do not consider the usual term structure of justification logics. Also note that there are topological approaches to evidence available~\cite{Baltag2016JustifiedBA,Benthem2012EvidenceLA,Benthem2014EvidenceAP}, which, however, do not feature justifications explicitly in their language.

It is the aim of this paper to provide a new semantics, called \emph{subset semantics} for justification logic that interprets terms as sets of possible worlds and operations on terms as operations on sets of possible worlds. We will then say that $t:A$ is true if $A$ is true in all worlds belonging to the interpretation of $t$.
We give a systematic study of this new semantics including soundness and completeness results and we show that the approach of~\cite{artemov2016onAggregatingPE} can be seen as a special case of our semantics.

Usually, justification logic includes an application operator that represents modus ponens (MP) on the level of terms.
We provide two approaches to handle this operator in our semantics. The first is to include a new constant $\cstar$, which is interpreted as the set of all worlds closed under (MP) and then use this new constant to define an application operator. 
The second way is to include a 
application operator directly. However, this leads to some quite cumbersome definitions.

Another difference between our semantics and many other semantics for justification logic is that we allow non-normal (impossible) worlds. They are usually needed to model the fact that agents are not omniscient and that they do not see all  consequences of the facts they are already aware of.
In an impossible world both $A$ and $\neg A$ may be true or none of them. This way of using impossible worlds was investigated by Veikko Rantala~\cite{rantala1982a,rantala1982b}. 

We start with presenting the $\cstar$-subset models with the corresponding syntax, axioms and semantics and proving soundness and completeness. In a second part we will present the alternative approach, i.e.~keeping the (j)-axiom and dealing with some cumbersome definitions within the semantics. It will be shown that the corresponding models are sound and complete as well. In a last section we will show that $\cstar$-subset models can be used to reason about uncertain knowledge by referring to Artemov's work on aggregating probabilistic evidence.

\section{$\LJCS$-subset models}

\subsection{Syntax}
Justification terms are built from countably many constants $c_i$ and variables $x_i$ and the special and unique constant $\cstar$ according to the following grammar:
\[
t::=c_i \ |\  x_i\ | \ \cstar\ | \ (t+t) \ | \ !t 
\]
The set of terms is denoted by $\Tm$. The set of atomic terms, i.e. terms that do not contain any operator $+$ or $!$ is denoted by $\ATm$.
The operation $+$ is left-associative. \\
Formulas are built from countably many atomic propositions $p_i$ and the  symbol~$\perp$ according to the following grammar:
\[
F::=p_i \ | \ \perp \ | \ F\to F \ | \ t:F
\]
The set of atomic propositions is denoted by $\Prop$ and the set of all formulas is denoted by $\LJ$. The other classical Boolean connectives $\neg,\top, \land,\lor,\leftrightarrow$ are defined as usual. 

%
\begin{defn}[$\cstarterm$] A $\cstarterm$ is defined inductively as follows:
	\begin{itemize}
		\item  $\cstar$ is a $\cstarterm$
		\item if $s$ and $t$ are terms and $c$ is a $\cstarterm$ then $s+c$ and $c+t$ are $\cstarterms$
	\end{itemize}
\end{defn}
So a $\cstarterm$ is either $\cstar$ itself or a sum-term where $\cstar$ occurs at least once.

%
We investigate a family of justification logics that differ in their axioms and how the axioms are justified.  We have two sets of axioms, the first axioms are:
\begin{fleqn}
\begin{equation}
\begin{array}{ll}\nonumber
\axcl & \text{all axioms of classical propositional logic};\\
\axjplus & s:A\lor t:A\to (s+t): A;\\
\axjcstar & c:A\land c:(A\to B) \to c:B \quad\text{ for all }\cstarterms \  c.
\end{array}
\end{equation}
\end{fleqn}
The set of these axioms is denoted by $\Lstarcomp$. \\
There is another set of axioms:
\begin{fleqn}
	\begin{equation}
	\begin{array}{ll}\nonumber
	\axjfour & t:A\to !t:(t:A);\\
	\axjd & t:\perp\to\perp;\\
	\axjt & t:A\to A.
	\end{array}
	\end{equation}
\end{fleqn}
This set is denoted by $\Lstarvol$.
It is easy to see that \axjd \ is a special case of \axjt.
By~$\LJstar$ we denote all logics that are composed from the whole set $\Lstarcomp$ and some subset of $\Lstarvol$.
Moreover, a justification logic $\LJstar$ is defined by the set of axioms and its constant specification $\CS$ that determines which constant justifies which axiom. So the constant specification is a set
\[\CS\subseteq\{(c, A)\enspace|\enspace c\text{ is a constant and } A\text{ is an axiom of }\LJstar\}\]
In this sense $\LJCS$ denotes the logic $\LJstar$ with the constant specification $\CS$. To deduce formulas in $\LJCS$ we use a Hilbert system given by $\LJstar$ and the rules modus ponens:
\begin{prooftree}
	\AxiomC{$A$}
	\AxiomC{$A\to B$}
	\RightLabel{(MP)}
	\BinaryInfC{$B$}
\end{prooftree}
and axiom necessitation

\begin{prooftree}
	\AxiomC{}
	\RightLabel{(AN!)\quad $\forall n\in\mathbb{N}$, where $(c, A)\in\CS$}
	\UnaryInfC{$\underbrace{!...!}_{n}:\underbrace{!...!}_{n-1}:\ ...:\ !!c:\ !c:c:A$}
\end{prooftree}

\begin{defn}[axiomatically appropriate $\CS$]\label{def:cs_axiomatocally_appropriate}
	A constant specification $\CS$ is called \emph{axiomatically appropriate} if for each axiom $A$, there is a constant $c$ with $(c,A) \in \CS$.
\end{defn}

\subsection{Semantics}

\begin{defn}[$\LJCS$-subset models] \label{cstar-subset models}Given some logic $\LJstar$ and some constant specification $\CS$, then an $\LJCS$-subset model $\mathcal{M}=(W, W_0, V, E)$ is defined by:
\begin{itemize}
\item $W$ is a set of objects called worlds.
\item $W_0\subseteq W$ and $W_0\neq\emptyset$ .
\item $V: W\times\LJ\to \{0, 1\}$  such that for all $\omega\in W_0$, $t\in\Tm$, $F, G\in\LJ$:
	\begin{itemize}
	\item $V(\omega, \perp)=0$;
	\item $V(\omega, F\to G)=1\quad\text{ iff }\quad V(\omega,F)=0$ or $V(\omega, G)=1$;
	\item\label{condition_tF} $V(\omega, t:F)=1\quad\text{ iff }\quad E(\omega, t)\subseteq\Set{\upsilon\in W | V(\upsilon, F)=1}$.
	\end{itemize}
\item $E: W\times\Tm \to \mathcal{P}(W)$ that meets the following conditions
where we use 
\begin{equation}\label{eq:truthset:1}
[A]:=\{\omega\in W\enspace|\enspace V(\omega, A)=1\}.
\end{equation}
For all $\omega\in W_0$, and for all $s, t\in\Tm$:
	\begin{itemize}
	\item $E(\omega, s+t)\subseteq E(\omega, s)\cap E(\omega, t)$;
	\item $E(\omega, \cstar)\subseteq\WMP$ where $\WMP$ is the set of deductively closed worlds, see below;
	\item if  \axjd~$\in\LJstar$, then $\exists\upsilon\in W_0$ with $\upsilon\in E(\omega, t)$;
	\item if \axjt~$\in\LJstar$, then $ \omega\in E(\omega, t)$;
	\item if \axjfour~$\in\LJstar$, then 
	\begin{multline*}
E(\omega, !t)\subseteq \\ \Set{\upsilon\in W|\forall F\in\LJ \ (V(\omega, t:F)=1\Rightarrow V(\upsilon, t:F)=1)};
\end{multline*}
	\item for all $n\in\mathbb{N}$ and for all $(c, A)\in\CS: E(\omega, c)\subseteq[A]$ and 
	\[E(\omega, \underbrace{!...!}_{n}c)\subseteq[\underbrace{!...!}_{n-1}c:....!c:c:A].\]

	\end{itemize}
\end{itemize}
The set $\WMP$ is formally defined as follows:
\begin{align*}
\WMP:=\{\omega\in W\ | \ \forall A, B\in\LJ \ (&(V(\omega, A)=1\text{ and }
V(\omega, A\to B)=1)\\ &\text{ implies } V(\omega, B)=1)\}.
\end{align*}
\end{defn}
So $\WMP$ collects all the worlds where the valuation function is closed under modus ponens. 
$W_0$ is the set of \emph{normal} worlds. The set $W \setminus W_0$ consists of the \emph{non-normal} worlds. 
Moreover, using the notation introduced by \eqref{eq:truthset:1}, we can read the condition on $V$ for justification terms $t:F$ as:
\[V(\omega, t:F)=1\quad\text{ iff }\quad E(\omega, t)\subseteq[F]\]

Since the valuation function $V$ is defined on worlds and formulas, the definition of truth is pretty simple:

\begin{defn}[Truth in $\LJCS$-subset models]\label{cstar-truth} Given an $\LJCS$-subset model $\mathcal{M}=(W, W_0, V, E)$ and a world $\omega\in W$ and a formula $F$ we define the relation $\Vdash$ as follows:
\[\mathcal{M}, \omega\Vdash F\quad\text{ iff }\quad V(\omega, F)=1\]
\end{defn}

\begin{remark}
	With the conditions on $E(w,\cstar)$ and $E(w,s+t)$ we obtain the intended meaning of a $\cstarterm \ s+\cstar$, namely that we consider only deductively closed worlds of $s$. 
	However, the set $E(s+\cstar)$ does not have to be exactly the intersection of $E(w, s)$ with $\WMP$ since we only have a subset-relation instead of equality. Hence $E(w, s+\cstar)\neq E(w,\cstar+s)$ in general. So even if in two $\cstarterms$ the exactly same evidence sets occur, their order still matters.	 
	For the same reason $s+t:A\to t+s:A$ is not valid for any two distinct terms $s$ and $t$. 
\end{remark}

\subsection{Soundness}
Since non-normal worlds will not be sound even with respect to the axioms of classical logic, we only have soundness within $W_0$.
\begin{thm}[Soundness of $\LJCS$-subset models]\label{cstar-soundness} For any justification logic $\LJCS$ and any formula $F\in\LJ$:
\[\LJCS\vdash F\quad\Rightarrow\quad\mathcal{M}, \omega\Vdash F\quad\text{ for all }\LJCS\text{-subset models }\mathcal{M}\text{ and all }\omega\in W_0\]
\end{thm}
\begin{proof}
The proof is by induction on the length of the derivation of $F$:
\begin{itemize}
\item If $F$ is an instance of some axiom of classical logic, then the truth of $F$ only depends on the valuation functions within the worlds of $W_0$. And all worlds of $W_0$ behave appropriately by definition.
\item If $F$ is derived by modus ponens, then there is a $G\in\LJ$ s.t.~$\LJCS\vdash G\to F$ and $\LJCS\vdash G$. By induction hypothesis $\mathcal{M}, \omega\Vdash G\to F$ hence \[V(\omega, G\to F)=1\] and therefore since $\omega\in W_0,\quad V(\omega, G)=0$ or $V(\omega, F)=1$ and again by induction hypothesis $\mathcal{M}, \omega\Vdash G$ and therefore $V(\omega, G)=1$. Because of this and $\omega\in W_0$, we obtain $V(\omega, F)=1$, which is $\mathcal{M}, \omega\Vdash F$. 
\item If $F$ is derived by axiom necessitation, then $F=c:A$ for some $(c, A)\in\CS$. By the condition on $E$ within $\LJCS$-subset models we have $E(\omega, c)\subseteq[A]$ for all $\omega\in W_0$. Hence $V(\omega, c:A)=1$  and therefore $\mathcal{M}, \omega\Vdash c:A$. If $F$ is a more complex formula like $!c:(c:A)$ derived by axiom necessitation, the argument is analogue.
\item If $F$ is an instance of the \axjplus-axiom, then $F=s:A\lor t:A\to s+t:A$ for some $s, t\in\Tm$ and $A\in\LJ$.\\
Suppose wlog.~$\mathcal{M}, \omega\Vdash s:A$, by Definition \ref{cstar-truth} we get $V(\omega, s:A)=1$ and by Definition \ref{cstar-subset models} and the conditions on $V$ for worlds in $W_0$,  $E(\omega, s)\subseteq[A]$. Since $E(\omega, s+t)\subseteq E(\omega, s)\cap E(\omega, t)\subseteq E(\omega, s)$ we obtain that $E(\omega, s+t)\subseteq[A]$ and by the condition on $E$ in $W_0$ in Definition \ref{cstar-subset models} that $V(\omega, s+t:A)=1$. Hence by Definition \ref{cstar-truth} $\mathcal{M}, \omega\Vdash s+t:A$.
%
\item If $F$ is an instance of the \axjcstar-axiom, then \[F=c:A\land c:(A\to B)\to c:B\] for some $A, B\in \LJ$ and a $\cstarterm \ c$.\\
Suppose that $\model,\omega\Vdash c:A$ and $\model,\omega\Vdash c:(A\to B)$. i.e.~$E(\omega,c)\subseteq[A]$ and $E(\omega,c)\subseteq[A\to B]$. Hence for all $\upsilon\in E(\omega,c)$ we obtain $V(\upsilon,A)=1$ and $V(\upsilon, A\to B)=1$.
From the definition of $\cstarterms$, the conditions on $E(w,\cstar)$ and $E(w,s+t)$ for some terms $s,t$, we infer that $E(w,c)\subseteq \WMP$ and we conclude $V(\upsilon,B)=1$ and hence $E(\omega,c)\subseteq[B]$ and this means that $\model, \omega\Vdash c:B$.
\item If F is an instance of the \axjd-axiom, then $F=t:\perp\to\perp$ for some $t\in\Tm$.\\
Suppose  towards a contradiction that $\mathcal{M}, \omega\Vdash t:\perp$ for some $t\in\Tm$, then by Definition \ref{cstar-truth} we obtain that $V(\omega, t:\perp)=1$ and hence by the condition of $E$ in the worlds of $W_0, \quad E(\omega, t)\subseteq[\perp]$. Since $\mathcal{M}$ must be a \axjd-$\LJCS$-subset model we claim that $\exists \upsilon \in W_0$ s.t.~$\upsilon\in E(\omega, t)$. From $\upsilon\in E(\omega, t)$ we derive by the condition on $V$ in Definition \ref{cstar-subset models} $\upsilon\in[\perp]$ or in other words $\upsilon\in(\upsilon'\in W\enspace|\enspace V(\upsilon',\perp)=1)$ and hence $V(\upsilon, \perp)=1$ and this contradicts the claim that $\upsilon\in W_0$.
\item If $F$ is an instance of the \axjt-axiom, then $F=t:A\to A$ for some $A\in\LJ$ and some $t\in\Tm$.\\
Suppose $\mathcal{M}, \omega\Vdash t:A$. By Definition \ref{cstar-truth} we obtain that $V(\omega, t:A)=1$. By the condition on worlds in $W_0$ in Definition \ref{cstar-subset models} we get $E(\omega, t)\subseteq[A]$. Since $\mathcal{M}$ is a \axjt-$\LJCS$-subset model, $\omega\in E(\omega, t)$ and therefore we conclude $\omega\in[A]$. Hence $V(\omega, A)=1$ and by Definition \ref{cstar-truth} we obtain that $\mathcal{M}, \omega\Vdash A$.
\item If $F$ is an instance of the \axjfour-axiom, then $F=t:A\to !t:(t:A)$ for some $A\in\LJ$ and $t\in\Tm$\\
Suppose $\mathcal{M}, \omega\Vdash t:A$, then by Definition \ref{cstar-truth} we obtain that $V(\omega, t:A)=1$. By the condition on $E$ for \axjfour-$\LJCS$-subset models for all $\upsilon\in E(\omega, !t)$ we obtain $V(\upsilon, t:A)=1$. Therefore $E(\omega, !t)\subseteq[t:A]$ and by Definition~\ref{cstar-subset models} there is $V(\omega, !t:(t:A))=1$ and again by Definition~\ref{cstar-truth} we conclude $\mathcal{M}, \omega\Vdash\ !t:(t:A)$.
\qedhere
\end{itemize}
\end{proof}

The  \axj-axiom $s:(A\to B)\to(t:A\to s\cdot t:B)$ is not part of our logic. 
Using the $(\cstar)$-axiom, we can define an application operation such that  the \axj-axiom is valid.

\begin{defn}[Application]\label{application} We introduce a new abbreviation $\cdot$ on terms by:
\[s\cdot t:=s+t+\cstar\]
\end{defn}

\begin{lem}[The ``$\textbf{j}$-axiom'' follows]\label{cstar-j-follows} For all $\model=(W, W_0, V, E)$, $\omega\in W_0$, $A, B\in\LJ$ and $s,t\in\Tm$:
\[\model,\omega\Vdash s:(A\to B)\to(t:A\to s\cdot t:B)\]
\end{lem}
\begin{proof} Assume $\model,\omega\Vdash s:(A\to B)$ and $\model, \omega\Vdash t:A$. Thus $E(\omega, s)\subseteq[A\to B]$ and $E(\omega, t)\subseteq[A]$. We find
\begin{multline*}
E(\omega, s\cdot t)= E(\omega, s+t+\cstar)=\\E(\omega, s)\cap E(\omega, t)\cap E(\omega,\cstar)\subseteq[A\to B]\cap[A]\cap E(\omega,\cstar). 
\end{multline*}
Hence for all $\upsilon\in E(\omega, s\cdot t)$ we have $V(\upsilon, A\to B)=1$ and $V(\upsilon, A)=1$ and $\upsilon\in E(\omega,\cstar)$ and therefore $V(\upsilon, B)=1$. Hence $E(\omega, s\cdot t)\subseteq[B]$ and we obtain $\model,\omega\Vdash s\cdot t:B$.
\end{proof}

Of course there is as well a derivation within any of the presented logics.
We use CR as an abbreviation for classical reasoning.
\begin{align*}
	&s:(A\to B)\to s+t:(A\to B)&\axjplus\\ 
	&s+t:(A\to B)\to s+t+\cstar:(A\to B)&\axjplus\\
	&s:(A\to B)\to s+t+\cstar:(A\to B)&\text{CR }\\
	&t:A\to s+t:A&\axjplus\\
	&s+t:A\to s+t+\cstar:A&\axjplus\\
	&t:A\to s+t+\cstar:A&\text{CR}\\
	&s+t+\cstar:(A\to B)\to(s+t+\cstar:A\to s+t+\cstar:B)&\axjcstar\\
	&s:(A\to B)\to(t:A\to s+t+\cstar:B)&\text{CR}
\end{align*}
%


\subsection{Completeness}
To prove completeness we will construct a canonical model and then show that for every formula $F$ that is not derivable in $\LJCS$, there is a model $\model^C$ with a world $\Gamma\in \WnullC$ s.t.~$\model^C,\Gamma\Vdash \neg F$. 
Like in the case of other semantics for justification logics, the completeness of logics containing ($\axjd$) is only given, if the corresponding constant specification is axiomatically appropriate. 
Before we start with the definition of the canonical model, we must do some preliminary work. We will first prove that our logics are conservative extensions of classical logic. With this result we can argue, that the empty set is consistent and hence can be extended to so-called maximal $\LJCS$-consistent sets of formulas. These sets will be used to build the $W_0$-worlds in the canonical model.
\begin{thm}[Conservativity]\label{cstar-conservativity} All  logics $\LJstar$ presented are conservative extensions of the classical logic $\CL$, i.e.~for any formula $F\in\Lcp$:
\[\LJstar\vdash F\quad\Leftrightarrow\quad\CL\vdash F\]
\end{thm}
\begin{proof}
	Since $\LJstar$ is an extension of $\CL$ the right-to-left direction is obvious.\\
	To prove the direction from  left to right we use a translation $\mathfrak{t}:\LJ\to\Lcp$:
	\begin{align*}
		\mathfrak{t}(P)&:=P\\
		\mathfrak{t}(\perp)&:=\perp\\
		\mathfrak{t}(A\to B)&:=\mathfrak{t}(A)\to \mathfrak{t}(B)\\
		\mathfrak{t}(s:A)&:=\mathfrak{t}(A)
	\end{align*}
	This translation removes all justification terms from a given formula. 
	Now we show by induction on the length of the derivation for some formula $A$ that $\CL\vdash \mathfrak{t}(A)$ whenever $\LJstar\vdash A$ and note that $\mathfrak{t}(A)=A$ for any $A\in\Lcp$. The cases where $A$ is an axiom of $\CL$ is then obvious, since all logics $\LJstar$ contain all axioms of $\CL$.
	\begin{itemize}
		\item \axcl: If $A$ is an instance of some axiom scheme in $\LJ$, then $\mathfrak{t}(A)=A$ is an instance of the same axiom scheme in $\CL$.
		\item \axjplus: $\mathfrak{t}(s:A\lor t:A\to (s+t):A)=A\lor A\to A$, which is a classical tautology. 
%
		\item \axjcstar: $\mathfrak{t}(c:A\land c:(A\to B)\to c:B)=A\land (A\to B)\to B$, which is a classical tautology.
		\item \axjfour,\axjd,\axjt: All translations have the form $A\to A$, which is a classical tautology. 
		\item \textbf{modus ponens}: If $A$ is derived by modus ponens, then there is a formula $B$ s.t.~$\LJstar\vdash B\to A$ and $\LJstar\vdash B$ and by induction hypothesis \[\Lcp\vdash \mathfrak{t}(B)\to \mathfrak{t}(A)\] and $\Lcp\vdash  \mathfrak{t}(B)$ and hence $ \mathfrak{t}(A)$ can be derived in $\CL$ by modus ponens.
		\item \textbf{axiom necessitation}: If $A$ is derived by axiom necessitation, then $A$ is of the form $c:B$ for some axiom $B$. But $\mathfrak{t}(c:B)=B$ and $B$ is an axiom.
\qedhere
	\end{itemize}
\end{proof}
\begin{defn}[Consistency]\label{cstar-Consistency}
	A logical theory $\Logic$ is called consistent, if 
	$\Logic\not\vdash\perp$.
	A set of formulas $\Gamma\subset\LJ$ is called $\Logic$-consistent if $\Logic\not\vdash\bigwedge\Sigma\to\perp$ for every finite $\Sigma\subseteq\Gamma$.
	A set of formulas $\Gamma$ is called maximal $\Logic$-consistent, if it is $\Logic$-consistent and none of its proper supersets is.
\end{defn}

Since all presented logics are conservative extensions of $\CL$ , we have the following consistency result.
\begin{lem}[Consistency of the logics]\label{cstar-L_are-consistent}
	All presented logics are consistent.
\end{lem}

	As usual, we  have a Lindenbaum lemma and the usual properties of maximal consistent sets hold, see, e.g.,~\cite{justificationLogic}. 
\begin{lem}[Lindenbaum Lemma]\label{Lindenbaum} Given some logic $\Logic$, then for each $\Logic$-consistent set of formulas $\Gamma\subset\LJ$ there exists a maximal consistent set $\Gamma'$ such that $\Gamma\subseteq\Gamma'$.
\end{lem}

\begin{lem}[Properties of maximal consistent sets]\label{cstar-maxConsist}
	Given some logic $\Logic$ and its language $\LJ$. If $\Gamma$ is a maximal $\Logic$-consistent set, then for all $F,G\in\LJ$:
	\begin{enumerate}
		\item if $\Logic\vdash F$, then $F\in\Gamma$;
		\item $F\in\Gamma$ if and only if $\neg F\not\in\Gamma$;
		\item $F\to G\in\Gamma$ if and only if $F\not\in\Gamma$ or $G\in\Gamma$;
		\item $F\in\Gamma$ and $F\to G\in\Gamma$ imply $G\in\Gamma$.
	\end{enumerate}
\end{lem}


\begin{defn}[Canonical Model]\label{cstar-canonical_model} For a given logic $\LJCS$ 
	we define the canonical model $\model^C=(W^C, \WnullC, V^C, E^C)$ by:
\begin{itemize}
\item $W^C=\mathcal{P}(\LJ)$.
\item $\WnullC=\Set{\Gamma\in W^C| \Gamma\text{ is maximal }\LJCS-\text{consistent set of formulas}}$.
\item $V^C: V^C(\Gamma, F)=1\quad\text{ iff }\quad F\in\Gamma$;
\item $E^C:$ With $\Gamma/t:=\{F\in\LJ\enspace|\enspace t:F\in\Gamma\}$ and
 \begin{align*}
 \WMPC:=\Set{\Gamma\in W^C| \forall A,B\in\LJ: \text{ if }A\to B\in\Gamma\text{ and }A\in\Gamma \text{ then }B\in\Gamma}
\end{align*}
we define :
\begin{align*}
E^C(\Gamma, t)&=\Set{\Delta\in \WMPC|\Delta\supseteq \Gamma/t} \text{ if } t \text{ is a }\cstarterm\\
E^C(\Gamma,t)&=\Set{\Delta\in W^C|\Delta\supseteq\Gamma/t}\text{ otherwise}.
\end{align*}
\end{itemize}
\end{defn}
Now we must show that the  canonical model is indeed an $\LJCS$-subset model.
\begin{lem}
	\label{cstar-cm_is_ssn} 
	The canonical model $\mathcal{M}^C$ is an $\LJCS$-subset model
	if either 
	\begin{enumerate}
		\item $(\axjd)\notin\LJCS$ or
		\item the constant specification $\CS$ is axiomatically appropriate or $(\axjt)\in\LJCS$.
	\end{enumerate}
\end{lem}
\begin{proof} In order to prove this, we have to show that $\model^C$ meets all the conditions we made for the valuation and evidence function and the constant specification i.e.:
\begin{enumerate}
\item\label{cstar-w0_notEmpty} $\WnullC\neq\emptyset$.
\item\label{cstar-\WnullC} For all $\Gamma\in \WnullC$:
	\begin{enumerate}
	\item\label{cstar-not_perp} $V^C(\Gamma, \perp)=0$;
	\item\label{cstar-implication} $V^C(\Gamma, F\to G)=1\quad\text{ iff }\quad V^C(\Gamma,F)=0$ or $V^C(\Gamma, G)=1$;
	\item\label{cstar-justification} $V^C(\Gamma, t:F)=1\quad\text{ iff }\quad E^C(\Gamma, t)\subseteq[F]$.
	\end{enumerate}

\item\label{E} For all $\Gamma\in \WnullC, F\in\LJ,  s, t\in\Tm$:
	\begin{enumerate}
	\item $E^C(\Gamma, s+t)\subseteq E^C(\Gamma, s)\cap E^C(\Gamma, t)$;
	\item $E^C(\Gamma, \cstar)\subseteq \WMPC$;
	\item\label{cstar-jd} if \axjd~in $\LJstar$: $\forall \Gamma\in \WnullC$ and $\forall t\in\Tm: \exists\Delta\in \WnullC$ s.t.~$\Delta\in E^C(\Gamma, t)$;
	\item\label{cstar-jt} if \axjt~in $\LJstar$: $\forall \Gamma\in \WnullC$ and $\forall t\in\Tm: \Gamma\in E^C(\Gamma, t)$;
	\item\label{cstar-j4} if \axjfour~in $\LJstar$: 
	\begin{multline*}
	\hspace{-10pt}E^C(\Gamma, !t)\subseteq \\ \Set{\Delta\in W^C | \forall F\in\LJ(V^C(\Gamma, t:F)\Rightarrow V^C(\Delta, t:F)=1)};
	\end{multline*}
	\item\label{cstar-CS} for all $(c, A)\in\CS$: $E^C(\Gamma, c)\subseteq[A]$  and \[E^C(\Gamma, \underbrace{!...!}_{n}c)\subseteq[\underbrace{!...!}_{n-1}c:....!c:c:A]\text{ for all }n\in\mathbb{N}.\]
	\end{enumerate}
\end{enumerate}
So the proofs are here:
\begin{enumerate}
\item Since the empty set is proven to be $\LJCS$-consistent (see  Lemma \ref{cstar-L_are-consistent}) it can be extended by the Lindenbaum Lemma to a maximal $\LJCS$-consistent set of formulas $\Gamma$ with $\Gamma\in \WnullC$.
\item Suppose $\Gamma\in W^C_0$:
	\begin{enumerate}
	\item We claim $V^C(\Gamma, \perp)=0$: Suppose the opposite, then $V^C(\Gamma,\perp)=1$ hence by the definition of $V^C$ follows that $\perp\in \Gamma$. But this is a contradiction to the fact that $\Gamma$ is consistent.
	\item From left to right: Suppose $V^C(\Gamma, F\to G)=1$, then by the definition of $V^C, F\to G\in \Gamma$. Since $\Gamma$ is maximal $\LJCS$-consistent this implies by  Lemma \ref{cstar-maxConsist} (3) that $F\not\in\Gamma$ or $G\in\Gamma$. Hence again by the definition of $V^C, V^C(\Gamma, F)=0$ or $V^C(\Gamma, G)=1$.\\
	From right to left: Suppose $V^C(\Gamma, F)=0$ or $V^C(\Gamma, G)=1$, then by the definition of $V^C$ either $F\not\in\Gamma$ or $G\in\Gamma$. Since $\Gamma\in \WnullC, \Gamma$ is maximal $\LJstar$-consistent and hence in both cases by Lemma \ref{cstar-maxConsist} (3) $F\to G\in\Gamma$. But this means again by the definition of $V^C$ that $V(\Gamma, F\to G)=1$.
	\item From left to right: Suppose $V^C(\Gamma, t:F)=1$, then by Definition \ref{cstar-canonical_model} \ $t:F\in\Gamma$. Hence with the definition of $\Gamma/t$ we obtain $F\in\Gamma/t$. So for each $\Delta\in E^C(\Gamma, t), F\in\Delta$ (again by Definition \ref{cstar-canonical_model}). Hence for these $\Delta$ it follows by the definition of $V^C$ that $V^C(\Delta, F)=1$ and therefore $\Delta\in[F]$. Since this is true for all $\Delta\in E^C(\Gamma, t)$ we obtain $E^C(\Gamma, t)\subseteq[F]$.\\
	From right to left: The proof is by contraposition. \\Suppose $V^C(\Gamma, t:F)\neq 1$, then by the definition of $V^C\quad t:F\not\in\Gamma$. We define a world $\Delta$ by $\Delta:=\Gamma/t$. Since $\Delta\in\powerset(\LJ)$ we can be sure that $\Delta$ exists, i.e. $\Delta\in W$. Since $t:F\not\in\Gamma$ it follows that $F\not\in\Gamma/t$ and therefore $F\not\in\Delta$. But obviously $\Delta\supseteq\Gamma/t$ hence $\Delta\in E^C(\Gamma, t)$. So we conclude $E^C(\Gamma, t)\not\subseteq[F]$.\\
	It remains to show that in case $t$ is a $\cstarterm, \Delta:=\Gamma/t\in\WMPC$ since otherwise $\Delta\not\in E^C(\Gamma, t)$. In fact this is the case. Since $\Gamma\in \WnullC$ we obtain that $\Gamma$ is a maximal $\LJCS$-consistent set of formulas and hence, whenever $t:A, t(A\to B)\in\Gamma$ for a $\cstarterm \ t$ then by $\axjcstar$ we obtain $t: B\in\Gamma$. This means that whenever $A\in\Delta$ and $A\to B\in\Delta$ then $B\in\Delta$. Hence $\Delta=\Gamma/t$ is closed under modus ponens and therefore $\Delta\in \WMPC$. So together with the former reasoning $\Delta\in E(\Gamma,t)$.
	\end{enumerate}
\item\label{cstar-CM_E} Suppose $\Gamma\in \WnullC$:
\begin{enumerate}
	\item Given some $F\in\LJ, s, t\in\Tm$:
	To prove this, we start by an observation  on the relation between the sets $\Gamma/(s+t)$ and $\Gamma/s$ for $\Gamma\in \WnullC$. If $s:A\in\Gamma$ then since $\Gamma$ is maximal $\LJCS$-consistent $s+t:A\in\Gamma$ hence $\Gamma/s\subseteq\Gamma/(s+t)$. With the same reasoning $\Gamma/t\subseteq\Gamma/(s+t)$. So if $\Delta\supseteq\Gamma/(s+t)$ then $\Delta\supseteq\Gamma/s$ and $\Delta\supseteq\Gamma/t$. Hence $E^C(\Gamma, s+t)\subseteq E^C(\Gamma, s)$ and $E^C(\Gamma, s+t)\subseteq E^C(\Gamma, t)$.\footnote{ Please note if either $s$ or $t$ is a $\cstarterm$ this only holds due to $E^C(\Gamma, s+t)$ being constrained to $\WMP$ by the fast that $s+t$ is a $\cstarterm$ too.} Therefore $E^C(\Gamma, s+t)\subseteq E^C(\Gamma, s)\cap E^C(\Gamma, t)$.
	\item\label{cstar-Ecstar_subset}  This follows directly from the fact that $\cstar$ is a $\cstarterm$ and the definition of $E^C(\Gamma,t)$ for $\cstarterms$.
	\item If \axjd~in $\LJstar$, either $\CS$ is axiomatically approporiate or $(\axjt)\in\LJstar$ too. 	 
	\begin{itemize}
		\item $\CS$ is axiomatically appropriate.
		
	For any $\Gamma\in \WnullC$ we obtain $\neg(t:\perp)\in\Gamma$. Hence $\perp\not\in\Gamma/t$. Suppose towards a contradiction that $\Gamma/t$ is not $\LJCS$-consistent, i.e.~there exist $A_1,\dots A_n\in\Gamma/t$ s.t.~
	\begin{equation}\label{eq:AsToPerp}
	A_1,\dots,A_n\vdash_{\LJCS} \perp.
	\end{equation} 
	This together with the construction of $\Gamma/t$ leads to $t:A_1,\dots,t:A_n\in\Gamma$. Since $\CS$ is axiomatically appropriate we can use \eqref{eq:AsToPerp} to infer $t:A_1,\dots, t:A_n\vdash_\LJCS s(t):\perp$, for some term $s(t)$ only based on $t$. Since $\Gamma$ is assumed to be maximally consistent we can use ($\axjd$) and apply modus ponens to infer $\perp\in\Gamma$ which contradicts the assumption that $\Gamma$ is consistent. 
	 Therefore $\Gamma/t$ is $\LJCS$-consistent and can be expanded by the Lindenbaum Lemma to a maximal $\LJCS$-consistent set $\Delta\supseteq\Gamma/t$ with $\Delta\in \WnullC$ and $\Delta\in E^C(\Gamma, t)$.
	 \item $(\axjt)\in\LJstar$:
	 
	 The claim is a direct consequence of property (\ref{cstar-jt}) (see next item).
	 
	 \end{itemize}
	\item Suppose for some $F\in\LJ, \Gamma\in W^C_0$ and $t\in\Tm$ that $F\in\Gamma/t$, i.e.~$t:F\in\Gamma$, since $\Gamma$ is maximal $\LJCS$-consistent and $t:F\to F$ is an instance of the \axjt-axiom, we conclude that $F\in\Gamma$. Since $F$ was arbitrary we obtain $\Gamma\supseteq\Gamma/t$ and hence $\Gamma\in E^C(\Gamma, t)$.
	\item Suppose for some $\Delta\in E^C(\Gamma, !t)$, hence $\Delta\supseteq\Gamma/!t$. Then assume for some arbitrary $F\in\LJ,\enspace V(\Gamma, t:F)=1$ i.e.~by Definition \ref{cstar-canonical_model} $t:F\in \Gamma$. Since $\Gamma$ is maximal $\LJCS$-consistent and $t:F\to !t:(t:F)$ is an instance of the \axjfour-axiom we obtain $!t:(t:F)\in \Gamma$ and hence $t:F\in\Gamma/!t$. But then $t:F\in\Delta$ and by Definition \ref{cstar-canonical_model} it follows that $V^C(\Delta, t:F)=1$. Since $F$ was an arbitrary formula and $\Delta$ an arbitrary world of $E^C(\Gamma, !t)$ we conclude that the condition holds.
	\item Suppose $(c, A)\in\CS$, then maximal $\LJCS$-consistency implies for all $\Gamma\in W^C_0$ that $c:A\in \Gamma$. Hence $A\in\Gamma/c$ and for all $\Delta\in E^C(\Gamma, c)$ we obtain $A\in\Delta$ and therefore $E^C(\Gamma, c)\subseteq[A]$. 
	
	Furthermore maximal $\LJCS$-consistency implies for all $\Gamma\in \WnullC$  by axiom necessitation that \[\underbrace{!...!}_nc:...:!c:c:A\in\Gamma.\] Hence \[\underbrace{!...!}_{n-1}c:...:!c:c:A\in \Gamma/\underbrace{!...!}_nc\] and for all $\Delta\in E^C(\Gamma, \underbrace{!...!}_nc)$ we obtain \[\underbrace{!...!}_{n-1}c:...:!c:c:A\in\Delta\] and therefore \[E^C(\Gamma, \underbrace{!...!}_nc)\subseteq [\underbrace{!...!}_{n-1}c:...:!c:c:A].\]
\qedhere
\end{enumerate}

\end{enumerate}
\end{proof}
Hence the canonical model is an $\LJCS$-subset model and we are nearly done. The Truth Lemma follows very closely:

\begin{lem}[Truth Lemma]\label{cstar-truth_Lemma} Let $\model^C=(W^C, \WnullC, E^C, V^C)$ be a canonical model, then for all $\Gamma\in \WnullC$:
\[\model^C, \Gamma\Vdash F\text{ if and only if }F\in\Gamma.\]
\end{lem}

\begin{proof}
\[
\model^C,\Gamma\Vdash F\xLongLeftRightArrow{\text{Def.}\ \ref{cstar-truth}} V^C(\Gamma, F)=1 \xLongLeftRightArrow{\text{Def.}\ \ref{cstar-canonical_model}} F\in\Gamma
.
\qedhere\]
\end{proof}

Hence each maximal $\LJCS$-consistent set is represented by some world in the canonical model and thus completeness follows directly:

\begin{thm}[Completeness]\label{cstar-completeness} 
	Given some logic $\LJstar$ and a constant specification $\CS$, 
	which is required to be axiomatically appropriate in case $(\axjd)\in\LJstar$ and $(\axjt)\notin\LJstar$,
	then
\[\model, \Gamma\Vdash F\enspace \text{ for all }\LJCS\text{-subset models }\model\text{ and for all } \Gamma\in W_0\Longrightarrow\LJCS\vdash F.\]
\end{thm}
\begin{proof}
The proof works with contraposition: Assume that $\LJCS\not\vdash F$. Then $\{\neg F\}$ is $\LJCS$-consistent and by the Lindenbaum Lemma contained in some maximal $\LJCS$-consistent world $\Gamma$ of the canonical model $\model^C$. Then $\model^C, \Gamma\not\Vdash F$.
\end{proof}

\section{$\LJCSapp$-subset models}
In this part we present an alternative definition of subset models for justification logic that directly interprets the application operator. Hence we work with the standard language of justification logic and we consider the  \axj-axiom instead of the axiom~$(\cstar)$.

\subsection{Syntax}

In this section, justification terms are built from constants $c_i$ and variables $x_i$ according to the following grammar:
\[t::=c_i \ | \ x_i \ | \ (t\cdot t) \ | \ (t+t) \ | \ !t\]
This set of terms is denoted by $\TmA$. 
The operations $\cdot$ and $+$ are left-associative and $!$ binds stronger than anything else.
Formulas are built from atomic propositions $p_i$ and the following grammar:
\[F::=p_i \ | \ \perp \ | \ F\to F \ | \ t:F\]
The set of atomic propositions is denoted by $\Prop$ and the set of all formulas is denoted by $\LJA$. Again we use the other logical connectives as abbreviations.

As in the first section, we investigate again a whole family of logics. They are arranged in two sets of axioms. The first set, denoted by $\Lappcomp$ contains the following axioms:
\begin{fleqn}
	\begin{equation}
	\begin{array}{ll}\nonumber
	\textbf{cl} & \text{all axioms of classical propositional logic};\\
	\axj & s:(A\to B)\to (t:A\to s\cdot t:B);\\
	\axjplus & s:A\lor t:A\to (s+t): A.
	\end{array}
	\end{equation}
\end{fleqn}
The other is identical to $\Lstarvol$ (modulo the different language) and contains:
\begin{fleqn}
	\begin{equation}
	\begin{array}{ll}\nonumber
	\axjfour & t:A\to !t:(t:A);\\
	\axjd & t:\perp\to\perp;\\
	\axjt & t:A\to A.
	\end{array}
	\end{equation}
\end{fleqn}
For the sake of uniformity we denote this set of axioms by $\Lappvol$.
By $\LogicAPP$ we denote all logics that are composed from the whole set $\Lappcomp$ and some subset of $\Lappvol$.

There are no differences between these logics and the ones of the former section except in case of application. Therefore we skip all the details already mentioned and proved before.\\
$\CS$ and $\LJCSapp$ are defined as before except that the corresponding logic has changed as mentioned. And deducing formulas in $\LJCSapp$ works the same as in the previous section.

\subsection{Semantics}

\begin{defn}[$\LJCSapp$-subset models]\label{APP-subset models} Given some logic $\LJCSapp$ then an $\LJCSapp$-subset model $\mathcal{M}=(W, W_0, V, E)$ is defined like an $\LJCS$-subset model where 
\[
E: W\times\TmA \to \mathcal{P}(W)
\]
meets the following condition for terms  of the form $s \cdot t$:
\[E(\omega, s\cdot t)\subseteq\{\upsilon\in W\enspace|\enspace \forall F\in\APP_\omega(s, t)(\upsilon\in[F])\},\]
where we use
\[\APP_\omega(s, t):=\{F\in\LJA \ |\enspace\exists H\in\LJA \text{ s.t.~}E(\omega, s)\subseteq[H\to F]\text{ and }E(\omega, t)\subseteq[H]\}.\]
\end{defn}

The set $\APP_\omega(s,t)$ contains all formulas that are colloquially said derivable by applying modus ponens to a formula justified by $s$ and a formula justified by~$t$.

Truth in an  $\LJCSapp$-subset models is defined as before.
\begin{defn}[Truth in $\LJCSapp$-subset models]\label{APP-truth} Let $\model=(W, W_0, V, E)$ be an $\LJCSapp$-subset model, then for a world $\omega\in W$ and a formula $F$ we define the relation $\Vdash$ as follows:
	\[\model, \omega\Vdash F\quad\text{ iff }\quad V(\omega, F)=1.\]
\end{defn}

\subsection{Soundness}

\begin{thm}[Soundness of $\LJCSapp$-subset models]\label{APP-soundness} For any justification logic $\LogicAPP$, any constant specification $\CS$ and any formula $F$:
	\[\LJCSapp\vdash F\quad\Rightarrow\quad\model, \omega\Vdash F\quad\text{ for all }\LJCSapp-\text{subset models }\model\text{ and all }\omega\in W_0.\]
\end{thm}
\begin{proof} The proof is by induction on the length of the derivation of $F$ and it is analogue to the proof of Theorem \ref{cstar-soundness}. The only thing that changes is the case, where $F$ is an instance of the $\axj$-axiom:\\
	Then $F=s:(A\to B)\to(t:A\to s\cdot t:B)$ for some $s, t \in\TmA$ and $A, B\in\LJA$.
	Assume for some $\omega\in W_0$ that $\mathcal{M}, \omega\Vdash s:(A\to B)$ and $\mathcal{M}, \omega\Vdash t:A$. Then by Definition \ref{APP-truth} $V(\omega, s:(A\to B))=1$. Hence since $\omega\in W_0$ we obtain $E(\omega, s)\subseteq[A\to B]$ and by the same reason $V(\omega, t:A)=1$ and $E(\omega, t)\subseteq[A]$. From the definition of $\APP_\omega(r, s)$ we conclude that $B\in\APP_\omega(s, t)$. So for all $\upsilon\in E(\omega, s\cdot t)$ we obtain by the requirements of $E$ that $V(\upsilon, B)=1$ hence $E(\omega, s\cdot t)\subseteq[B]$. From this, the fact that $\omega\in W_0$ and the requirements of $V$ in $W_0$ we obtain $V(\omega, s\cdot t:B)=1$, which is by Definition \ref{APP-truth} $\mathcal{M}, \omega\Vdash s\cdot t:B$.\qedhere

\end{proof}

\subsection{Completeness}
Before we start defining a canonical model, we have to do the same preliminary work for $\LJCSapp$ as we had to do in the previous section for $\LJCS$. Since the logics $\LJCS$ from the former section differ only in one axiom, i.e. $\axj$ replaces $\axjcstar$, we skip all the parts that are already done and focus on the changes that it brings about.

As before, we have a conservativity and consistency result.
\begin{thm}[Conservativity]\label{APP-conservativity} All  logics $\LogicAPP$ presented are conservative extensions of the classical logic $\CL$, i.e. for any formula $F\in\Lcp$:
	\[\LogicAPP\vdash F\quad\Leftrightarrow\quad\CL\vdash F.\]
\end{thm}
\begin{lem}[Consistency of $\LogicAPP$]\label{APP-consistency} All logics in $\LogicAPP$ are consistent.
\end{lem}

All other ingredients that we needed in the former section to define and further develop the canonical model were generally defined and proven and can be adopted without additional effort.

To prove completeness we define a canonical model as follows:
\begin{defn}[Canonical Model]\label{APP-canonical_model} For a given logic $\LogicAPP$ and a constant specification $\CS$
	 we define the canonical model $\model^C=(W^C, \WnullC, V^C, E^C)$ by:
	\begin{itemize}
		\item $W^C=\mathcal{P}(\LJA)$;
		\item $\WnullC=\{\Gamma\in W^C\enspace|\enspace \Gamma\text{ is maximal }\LJCSapp-\text{consistent set of formulas}\}$;
		\item $V^C: V^C(\Gamma, F)=1\quad\text{ iff }\quad F\in\Gamma$;
		\item $E^C: E^C(\Gamma, t)=\{\Delta\in W\enspace| \enspace\Delta\supseteq \Gamma/t\}$. 
	\end{itemize}
\end{defn}
Now we must show that such a canonical model is in fact a subset model.
\begin{lem}
	\label{APP-cm_is_ssn} 
	The canonical model $\model^C$ is an $\LJCSapp$-subset model 			
	if $\LogicAPP$ does not contain (\axjd) or contains it but the corresponding constant specification $\CS$ is axiomatically appropriate or $(\axjt)\in\LogicAPP$ too.
\end{lem}
\begin{proof} In order to prove that, we have to proceed in the same way as in the previous section, i.e.~showing that $\model^C$ meets all the conditions we made for the valuation and the evidence function as well as the constant specification.
	
		\begin{enumerate}
			\item\label{APP-w0_notEmpty} $W_0\neq\emptyset.$
			\item\label{APP-W_0} For all $\Gamma\in \WnullC$:
			\begin{enumerate}
				\item\label{APP-not_perp} $V^C(\Gamma, \perp)=0$;
				\item\label{APP-implication} $V^C(\Gamma, F\to G)=1\quad\text{ iff }\quad V^C(\Gamma,F)=0$ or $V^C(\Gamma, G)=1$;
				\item\label{APP-justification} $V^C(\Gamma, t:F)=1\quad\text{ iff }\quad E(\Gamma, t)\subseteq[F]$.
			\end{enumerate}
			\item\label{APP-E} For all $\Gamma\in \WnullC, F\in\LJA, s, t\in\TmA$:
			\begin{enumerate}
				\item\label{APP-application} $E^C(\Gamma, s\cdot t)\subseteq\{\Delta\in W^C\enspace|\enspace \forall F\in\APP_\Gamma(s, t)(\Delta\in[F])\}$;
				\item\label{APP-plus} $E^C(\Gamma, s+t)\subseteq E^C(\Gamma, s)\cap E^C(\Gamma, t)$;
				\item\label{APP-jd} if \axjd~in L: $\forall \Gamma\in \WnullC$ and $\forall t\in\TmA: \exists\Delta\in \WnullC$ s.t.~$\Delta\in E^C(\Gamma, t)$;
				\item\label{APP-jt} if \axjt~in L: $\forall \Gamma\in \WnullC$ and $\forall t\in\TmA: \Gamma\in E^C(\Gamma, t)$;
				\item\label{APP-j4} if \axjfour~in L: 
					\begin{multline*}
					E^C(\Gamma, !t)\subseteq\\ \{\Delta\in W^C\enspace|\enspace\forall F\in\LJA (V^C(\Gamma, t:F)\Rightarrow V^C(\Delta, t:F)=1)\}.
					\end{multline*}
			\end{enumerate}
			\item\label{APP-CS} For all $(c, A)\in\CS,\quad E^C(\Gamma)\subseteq[A]$ for all $\Gamma\in W_0$.
		\end{enumerate}
	
 Since the canonical model is defined in the same way as the one of $\LJCS$-subset models, the corresponding proofs can be reused (see Lemma \ref{cstar-cm_is_ssn}). Nevertheless, there is some difference. Instead of showing that $E^C(\Gamma,\cstar)\subseteq\WMP^C$ we have to show that $E^C(\Gamma, s\cdot t)\subseteq\{\Delta\in W^C\enspace|\enspace \forall F\in\APP_\Gamma(s, t)(\Delta\in[F])\}$. Assume that we are given $\Gamma\in W^C_0$, $F\in\LJA$, $s, t \in\TmA$.
Take any  $\Delta\in E^C(\Gamma, s\cdot t)$, i.e.~$\Delta\supseteq\Gamma/(s\cdot t)$. Hence for all~$F$ \/ s.t.~$s\cdot t:F\in\Gamma$ we know that $F\in\Delta$. Hence by the definition of~$V^C$, we have $V(\Delta, F)=1$ and therefore $\Delta\in[F]$. 

It remains to show: if $F\in\APP_\Gamma(s, t)$ then $s\cdot t:F\in \Gamma$. Suppose for some formula $F$ that $F\in\APP_\Gamma(s, t)$ then by definition of $\APP_\Gamma(s, t)$ we know that there is a formula $H$ s.t.~$E^C(\Gamma, s)\subseteq[H\to F]$ and $E^C(\Gamma, t)\subseteq[H]$. By using Lemma \ref{APP-cm_is_ssn} (\ref{APP-justification}) we conclude $V^C(\Gamma, s:(H\to F))=1$ and $V^C(\Gamma, t:H)=1$. Hence by the definition of $V^C$ we obtain $s:(H\to F)\in\Gamma$ and $t:H\in\Gamma$ and since $\Gamma$ is maximal $\LJCSapp$-consistent and $s:(H\to F)\to(t:H\to s\cdot t:F)$ is an instance of the $\axj$-axiom we conclude that $s\cdot t:F\in\Gamma$.\qedhere

\end{proof}

\begin{lem}[Truth Lemma]\label{APP-truth_Lemma} Let $\model^C=(W^C, \WnullC, E^C, V^C)$ be some canonical $\LJCSapp$-subset model, then for all $\Gamma\in W_0$:
	\[\model^C, \Gamma\Vdash F\text{ if and only if }F\in\Gamma.\]
\end{lem}
\begin{proof} 
	\[\model^C,\Gamma\Vdash F\xLongLeftRightArrow{\text{Def.}\ \ref{APP-truth}} V^C(\Gamma, F)=1 \xLongLeftRightArrow{\text{Def.}\ \ref{APP-canonical_model}} F\in\Gamma.\qedhere\]
\end{proof}

\begin{thm}[Completeness]\label{APP-completeness} Given some logic $\LogicAPP$ and a constant specification $\CS$, 
	which is axiomatically appropriate in case $(\axjd)\in\LogicAPP$,
	 then
	\[\model, \Gamma\Vdash F\enspace \text{ for all models }\model\text{ and for all } \Gamma\in W_0\Longrightarrow\LJCSapp\vdash F.\]
\end{thm}
\begin{proof}
	The proof is analogue to the one of Theorem \ref{cstar-completeness}.
\end{proof}


\section{Artemov's aggregated evidence and $\LJCS$-subset models}

Artemov \cite{artemov2016onAggregatingPE} considers the case in which we have a database, i.e.~a set of propositions $\Gamma=\{F_1,\dots F_n\}$ with some kind of probability estimates and in which we also have some proposition $X$ that logically follows from $\Gamma$. Then we can search for the best justified lower bound for the probability of $X$. He presents us a nice way to find this lower bound. To find it, he assumes probability events $u_1,\dots,u_n$, each of them supporting some proposition in $\Gamma$, i.e.~$u_i:F_i$, and calculates some aggregated evidence $e(u_1,\dots,u_n)$ for $X$ with them. The probability of $e$ then provides a tight lower bound for the probability of $X$.

The trick he uses is the following:
\begin{enumerate}
\item\label{ArtfirstStep} First he collects all subsets $\Delta_i$ of $\Gamma$ which support $X$, i.e.~$\Delta_i\vdash X$, and creates a new evidence $t_i$ from all the corresponding $u_{i_j}$ s.t. $u_{i_j}:F_{i_j}$ for each $F_{i_j}\in\Delta_i$.
\item\label{ArtSecondStep} In the second step he combines all these new pieces of evidence to a new evidence (the so-called aggregated evidence) that actually is the greatest evidence supporting $X$.
\end{enumerate}

The model he has in mind contains some evaluation in a probability space $(\Omega, \mathcal{F}, P)$ with a mapping $\star$ from propositions to $\Omega$ and evidence terms to $\mathcal{F}$ that meets some restrictions (for more details on this see \cite{artemov2016onAggregatingPE}). Step (\ref{ArtfirstStep}) is to create a new evidence $t_i$ for each $\Delta_i$ described above, which consists of the intersection of the corresponding $u_{i_j}\hspace{-3pt}\text{'s}$.
\[t_{i}:=\bigcap \{u_{i_j}\ | \ u_{i_j}\subseteq F_{i_j}^\star \text{ for some }F_{i_j}\in\Delta_i\}.\]
Step (\ref{ArtSecondStep}) then is to union all these pieces of evidence to a new so-called aggregated evidence:
\[\AEGX:=\bigcup\{t_i\ | \ t_i \text{ is an evidence for }X\text { obtained by step (\ref{ArtfirstStep})}\}.\]

On the syntactic side evidence terms are built from variables $u_1,\dots,u_n$, constants $0$ and $1$ and operations $\cap$ and $\cup$, where $st$ is used as an abbreviation for $s\cap t$. With this we can built a free distributive lattice $\mathcal{L}_n$ where $st$ is the meet and $s\cup t$ is the join of $s$ and $t$, $0$ is the bottom and $1$ the top element of this lattice. 
Moreover Artemov defines formulas in a usual way from propositional letters $p, q, r,\dots$ by the usual connectives and adds formulas of the kind $t:F$ where $t$ is an evidence term and $F$ a purely propositional formula.

The logical postulates of the logic of Probabilistic Evidence $\PE$ are:
\begin{enumerate}
	\item axioms and rules of classical logic in the language of $\PE$;
	\item $s:(A\to B)\to(t:A\to [st]:B)$;
	\item\label{axPE_suniont} $(s:A\land t:A) \to[s\cup t]:A$;
	\item $1:A$, where $A$ is a propositional tautology,\\
	$0:F$, where $F$ is a propositional formula;
	\item\label{axPEpreorder} $t:X\to s:X$, for any evidence terms $s$ and $t$ such that $s\preceq t$ in $\mathcal{L}_n$.
\end{enumerate}
Artemov presents Soundness and Completeness proofs connecting $\PE$ with the presented semantic, for more details see \cite{artemov2016onAggregatingPE}.

Before we can start adapting Artemovs approach to our models, we have to point out some differences between the semantics and syntax used. First, contrary to the models of Artemov, subset models may contain inconsistent worlds, but this does not significantly affect the applicability of Artemov's approach on them.

Another difference is that our evidence function has another domain. In Artemov's models the evidence functions is $E:\Tm\to\powerset(\Omega)$ while in our models it is $E:W\times\Tm\to\powerset(W)$. This difference is due to the fact that we allow terms to justify non-purely propositional formulas. Although we need to adapt Artemov's definitions, these adaptations will maintain the essential characteristics. So let's adapt the $\LJCS$-subset models to aggregated $\LJCS$-subset models
 by first describing the new syntax for the terms:
\begin{defn}[Justification Terms] 
 Justification terms are built from constants $0, 1, c_i$ and variables $x_i$ and the special and unique constant $\cstar$ according to the following grammar:
 \[t::=0 \ | \ 1\ | \ c_i\ | \ x_i\ | \ \cstar \ | \ (t+t) \ | \ (t\cup t) \ | \ !t\]
\end{defn}
This set of terms is denoted by $\TmP$. As before, we introduce the abbreviation $st:=s+t+\cstar$.\\
Even though we have other operators as well, we can construct a free distributive lattice where we take $s+t$ as the meet of $s$ and $t$, $s\cup t$ as the join of them, $0$ as the bottom element of the lattice. Note, that $st$ then is the meet of $s$, $t$, and~$\cstar$. Moreover, $1$ and $!t$ are treated like constants.\footnote{We do not claim that $1$ is the top element since some set $E(\omega, t)$ for a world $\omega\in W_0$ and $t\in\TmP$ may contain non-normal worlds. If we claimed that $1$ was the top element we would obtain $t\preceq 1$ and furthermore the set $E(\omega, 1)$ would contain non-normal worlds as well. But since in non-normal worlds axioms may not be true, $E(\omega, 1)\not\subseteq [A]$ for some axiom $A$ may be the case and therefore axiom (4) would fail.}  As usual, we have
\begin{equation}\label{eq:less:1}
s\preceq t  \quad\text{if{f}}\quad s \cup t = t
\end{equation}

There is no difference to our subset models regarding the rules for forming formulas except that the terms are contained in $\TmP$, of course. The set of formulas built according to these grammar and rules is denoted by $\LJprob$.

In the definition of $\LJCS$-subset models we only change the conditions on the evidence function and the domain of $V$. 
\begin{defn}[$\PE$-adapted $\LJCS$-subset models]\label{adapted_ssm}
An $\LJCS$-subset model is called a $\PE$-adapted $\LJCS$-subset model if the valuation function and the evidence function meet the additional conditions respectively are redefined as follows:\\
$V:W\times \LJprob\to\{0, 1\}$ where all conditions listed in Definition \ref{cstar-subset models} remain the same.\\
For all $\omega\in W_0$ and for all $s,t\in\TmP$
\begin{itemize}
	\item $E(\omega,1)=W_0$;
	\item $E(\omega,0)=\emptyset$;
	\item $E(\omega,s\cup t) = E(\omega,s)\cup E(\omega, t)$. 
\end{itemize}
\end{defn}

And in fact, such an $\PE$-adapted $\LJCS$-subset model is a model of probabilistic evidence $\PE$.

\begin{thm}[Soundness]
	$\PE$-adapted $\LJCS$-subset models $\model$ are sound with respect to probabilistic evidence $\PE$, i.e. for all $F\in\LJprob$
	\[\PE\vdash F\enspace \Rightarrow\enspace \model,\omega\Vdash F\enspace\text{ for all }\PE\text{-adapted }\LJCS\text{-subset models and all }\omega\in W_0.\]
\end{thm}
\begin{proof}
	The proof is by induction on the length of the derivation of $F$:
	\begin{itemize}
		\item If $F$ is derived by axiom necessitation or modus ponens or is an instance of axiom (1),
		then the proof is the analogue as in Theorem \ref{cstar-soundness} since the relevant definitions have remained the same.
		\item If $F$ is an instance of axiom (2) then the proof is analogue to the proof of Lemma \ref{cstar-j-follows}: Suppose $\model,\omega\Vdash s:(A\to B)$ and $\model, \omega\Vdash t:A$ then $E(\omega, s)\subseteq[A\to B]$ and $E(\omega, t)\subseteq[A]$. 
		\begin{multline*}
		E(\omega, st)=E(\omega, s+t+\cstar)\subseteq \\E(\omega, s)\cap E(\omega, t)\cap E(\omega,\cstar)\subseteq[A\to B]\cap[A]\cap E(\omega,\cstar).
		\end{multline*}
		 Hence for all $\upsilon\in E(\omega, st)$ we have $V(\upsilon, A\to B)=1$ and $V(\upsilon, A)=1$ and $\upsilon\in E(\omega,\cstar)$ and therefore $V(\upsilon, B)=1$. Hence $E(\omega, st)\subseteq[B]$ and we obtain $\model,\omega\Vdash st:B$.
		\item If $F$ is an instance of axiom (3) then $F=(s:A\land t:A)\to[s\cup t:A]$ for some $A\in\LJprob, s, t\in\TmP$. Suppose $\model,\omega\Vdash s:A\land t:A$ hence $E(\omega, s)\subseteq[A]$ and $E(\omega, t)\subseteq[A]$. Therefore $E(\omega, s\cup t) \subseteq E(\omega, s)\cup E(\omega, t)\subseteq[A]$ and since $\omega\in W_0$ we obtain $\model,\omega\Vdash s\cup t:A$.
		\item If $F$ is an instance of axiom (4) then either $F=1:A$ for some axiom $A$ or $0:G$ for some formula $G$.\\
		Suppose $F=1:A$ for some axiom $A$. We assume that $\model,\omega\Vdash A$ for all $\omega\in W_0$, hence $E(\omega, 1)=W_0\subseteq[A]$ and therefore $\model,\omega\Vdash 1:A$ for all $\omega\in W_0$.\\
		Suppose $F=0:G$: For any $\omega\in W_0$ we have $E(\omega, 0)=\emptyset$ by Definition \ref{adapted_ssm}. Since $\emptyset$ is a subset of any subset of $W$, we obtain $E(\omega, 0)=\emptyset\subseteq[G]$ for any formula $G\in\LJprob$.
		\item $F$ is an instance of axiom (5). Assume $\model, \omega\Vdash t:X$ for some term $t$ and some formula $X$ and let $s\preceq t$. By \eqref{eq:less:1} we find $t=s\cup t$. Thus 
		\[
		E(\omega, t)=E(\omega, s\cup t)= E(\omega, s)\cup E(\omega, t)
		\]
		  and therefore  $E(\omega, s)\subseteq E(\omega ,t)$. The assumption $\model,\omega\Vdash t:X$ means that $E(\omega, t)\subseteq[X]$. So we get $E(\omega, s)\subseteq[X]$ and conclude $\model,\omega\Vdash s:X$.\qedhere
	\end{itemize}
\end{proof}

\begin{thm}[model existence]
	There exists a $\PE$-adapted $\LJCS$-subset model.
\end{thm}
\begin{proof}
We construct a model $\model=\{W, W_0, V, E\}$ as follows:
	\begin{itemize}
		\item $W=W_0=\{\omega\}$.
		\item The valuation function is built bottom up:
		\begin{enumerate}
			\item $V(\omega, \perp)=0$;
			\item $V(\omega, P)=1$, for all $P\in\Prop$;	
			\item $V(\omega, A\to B)=1$ iff $V(\omega, A)=0$ or $V(\omega, B)=1$;
			\item $V(\omega, t:F)=1$ iff $t\not\geq 1$ or if $t\geq 1$ and $V(\omega, F)=1$.
		\end{enumerate} 
		\item $E(\omega, t)=\begin{cases}
		\{\omega\} & \text{ if } t \geq 1\\
		\emptyset & \text{ otherwise}.
		\end{cases}$
	\end{itemize}

It is straightforward to show that $\model$ is indeed a  $\PE$-adapted $\LJCS$-subset model.
Let us only show the condition $E(\omega,s\cup t) = E(\omega,s)\cup E(\omega, t)$.

Suppose first $s, t\not\geq 1$, Then $E(\omega, s\cup t)=\emptyset=E(\omega, s)=E(\omega, t)$ and hence the claim follows immediately.

Suppose at least one term of $s$ and $t$ is in greater than 1, then $E(\omega, s)=\{\omega\}$ or $E(\omega, t)=\{\omega\}$ and hence $E(\omega, s)\cup E(\omega,t)=\{\omega\}$ and since $s\leq s\cup t$ and $t\leq s\cup t$ we obtain $s\cup t\geq 1$ and therefore $E(\omega, s\cup t)=\{\omega\}$, so the claim holds.\qedhere
\end{proof}

Note that we cannot use the canonical model to show that adapted subset models exists since in the canonical model
\[
E(\Gamma, s\cup t)\not\subseteq E(\Gamma, s)\cup E(\Gamma, t).
\]
However, in an adapted model we need these sets to be equal (see Definition~\ref{adapted_ssm}) since otherwise axioms  (3) and (5) would not be sound.

\section{Conclusion}
We  introduced a new semantics, called subset semantics, for justifications. So far, often a symbolic approach was used to interpret justifications. In our semantics, justifications are modeled as sets of possible worlds. 
We also presented a new justification logic that is sound and complete with respect to our semantics. Moroever, we studied a variant of subset models that corresponds to traditional justification logic. 

Subset models provide a versatile tool to work with justifications. In particular, we can naturally extend them with probability measures to capture uncertain justifications. In the last part of the paper, we showed that subset models subsume Artemov's approach to aggregating probabilistic evidence.

\section*{Acknowledgements}
We would like to thank the anonymous referee who found a mistake in the original completeness proof for $\LJCS$.


\end{document}